\documentclass[9pt,twoside,reqno]{amsart}
\usepackage{amsfonts}
\usepackage{hyperref}

\newtheorem{theorem}{Theorem}
\theoremstyle{plain}

\newtheorem{definition}{Definition}

\newtheorem{remark}{Remark}

\numberwithin{equation}{section}
\input{tcilatex}

\begin{document}
\title[]{Generalizations of Banach and Kannan Fixed point theorems in $%
b_{v}\left( s\right) $ metric spaces}
\author{Ibrahim Karahan}
\address{Department of Mathematics, Faculty of Science, Erzurum Technical
University, Erzurum, 25700, Turkey.}
\email{ibrahimkarahan@erzurum.edu.tr}

\begin{abstract}
Generalizations of a metric space is one of the most important research
areas in mathematics. In literature ,there are several generalized metric
spaces. The latest generalized metric space is $b_{v}\left( s\right) $
metric space which is introduced by Mitrovic and Radenovic in 2017. In this
paper, we prove Kannan fixed point theorem and generalize Banach fixed point
theorem for weakly contractive mappings in $b_{v}\left( s\right) $ metric
spaces. Our results extend and generalize some corresponding result.
\end{abstract}

\maketitle

\section{Introduction and Preliminaries}

Let $\left( E,\rho \right) $ be a metric space and $S$ a mapping on $E$. If
there exists a point $u\in E$ such that $Su=u$, then point $u$ is said to be
a fixed point of $S$ and the set of all fixed points of $S$ is denoted by $%
F(S)$. Fixed point theory is one of the most important and famous theory in
mathematics since it has applications to very different type of problems
arise in different branches. So, uniqueness and existence problems of fixed
points are also important. Two of the well known fixed point theorems are
Banach and Kannan fixed point theorems. Banach fixed point theorem proved by
Stefan Banach in 1920 guarantees that a contractive mapping (a mapping $S$
is called contractive if there exists a $c\in \left[ 0,1\right) $ such that $%
\rho \left( Su,Sw\right) \leq c\rho \left( u,w\right) $ for all $u,w\in E)$
defined on a complete metric space has a unique fixed point, (see \cite{ban}%
). In 1968, Kannan proved another fixed point theorem for mapping satisfying 
\begin{equation}
\rho \left( Su,Sw\right) \leq \gamma \left( \rho \left( u,Su\right) +\rho
\left( w,Sw\right) \right)  \label{i1}
\end{equation}%
for all $u,w\in E$ and $\gamma \in \left[ 0,\frac{1}{2}\right) ,$ (see \cite%
{kan}). Although contractivity condition implies the uniform continuity of $%
S $, Kannan type mappings (mappings which satisfy the inequality (\ref{i1}))
need not to be continious. Also both of these theorems characterize the
completeness of metric spaces. Because of the importance of these theorems
many authors generalized them for different type of contractions (see \cite%
{shu, zey, suz, ens, mih, ak}). In one of these studies, Rhoades \cite{rho}
proved that a mapping which satisfies for all $u,w\in E$ 
\begin{equation}
\rho \left( Su,Sw\right) \leq \rho \left( u,w\right) -\varphi \left( \rho
\left( u,w\right) \right)  \label{i2}
\end{equation}
has a unique fixed point where $\left( E,\rho \right) $ is a complete metric
space and $\varphi :%
\mathbb{R}
^{+}\cup \left\{ 0\right\} \rightarrow 
\mathbb{R}
^{+}\cup \left\{ 0\right\} $ is a nondecreasing and continuous function such
that $\varphi (t)=0$ iff $t=0$. Here, $%
\mathbb{R}
$ is the set of all real numbers. A mapping which satisfies inequality (\ref%
{i2}) is called weakly contractive. It is clear that weakly contractive
mappings can be reduced to contractive mappings by taking $\varphi (t)=ct$
for $c\in \left[ 0,1\right) $.

It is well known that metric spaces are very important tool for all branches
of mathematics. So mathematicians have been tried to generalize this space
and transform their studies to more generalized metric spaces.

As one of the most famous generalized metric spaces, in 1989, $b$-metric
spaces was introduced by the following way.

\begin{definition}
\cite{Bakhtin} Let $E$ be a nonempty set and mapping $\rho :E\times
E\rightarrow \lbrack 0,\infty )$ a function. $\left( E,\rho \right) $ is
called $b$-metric space if there exists a real number $s\geq 1$ such that
following conditions hold for all $u,w,z\in E$:

\begin{enumerate}
\item[(1)] $\rho (u,w)=0$ iff $u=w$;

\item[(2)] $\rho (u,w)=\rho (w,u)$;

\item[(3)] $\rho (u,w)\leq s[\rho (u,z)+\rho (z,w)].$
\end{enumerate}
\end{definition}

Clearly a $b$-metric space for $s=1$ is exactly a metric space. After this
definition, many authors proved fixed point theorems for different type
mappings in this space (see \cite{AAR, A, ge, Na, hus, ros, aydi, kir}).
Also they gave some generalization of Banach contraction principle, Reich
and Kannan fixed point theorems. In this sense, Czerwik \cite{Cze, Cz}
generalized Banach contraction principle to $b$-metric spaces. Following $b$%
-metric spaces, some generalized version of this space such as extended $b$%
-metric space, dislocated $b$-metric space, recrangular $b$-metric space,
partial $b$-metric space, partial ordered $b$-metric space, etc. were
introduced. The latest generalized $b$-metric space was introduced by
Mitrovic and Radenovic \cite{mr} in 2017 by the following way.

\begin{definition}
\cite{mr} Let $E$ be a set, $\rho :$ $E\times E\rightarrow \lbrack 0,\infty
) $ a function and $v\in 
\mathbb{N}
$. Then $(E,\rho )$ is said to be a $b_{v}(s)$ metric space if there exists
a real number $s\geq 1$ such that following conditions hold for all $u,w\in
E $ and for all distinct points $z_{1},z_{2},...,z_{v}\in E$, each of them
different from $u$ and $w$:

\begin{enumerate}
\item[(1)] $\rho (u,w)=0$ iff $u=w$;

\item[(2)] $\rho (u,w)=\rho (w,u)$;

\item[(3)] $\rho (u,w)\leq s[\rho (u,z_{1})+\rho (z_{1},z_{2})+\cdots +\rho
\left( z_{v},w\right) ].$
\end{enumerate}
\end{definition}

$b_{v}\left( s\right) $ metric space generalizes not only $b$-metric space
but also rectangular metric space, $v$-generalized metric space and
recrangular $b$-metric space. Also, Mitrovic and Radenovic proved Banach
contraction principle and Reich fixed point theorem in this space.

Motivated and ispired by the above studies, in this paper, we generalize
Banach fixed point theorem for weakly contractive mappings and prove Kannan
fixed point theorem in $b_{v}\left( s\right) $ metric spaces.

\section{Main Results}

In this section we give generalizations of Banach and Kannan fixed point
theorem in $b_{v}\left( s\right) $ metric space.

\begin{theorem}
\label{A}Let $E$ be a complete $b_{v}\left( s\right) $ metric space and $S$
a weakly contractive mapping on $E$. Then $S$ has a unique fixed point.
\end{theorem}

\begin{proof}
Let $u_{0}\in E$ be an arbitrary initial point. Define sequence $\left\{
u_{n}\right\} $ by $u_{n+1}=Su_{n}$. If $u_{n}=u_{n+1}$, then it is clear
that $u_{n}$ is a fixed point of $S$. So, let us assume that $u_{n}\neq
u_{n+1}$ for all $n$. Moreover, for all different $n$ and $m$, assumption $%
x_{n}\neq x_{m}$ can be proved in a similar way to other studies. Then, from
definition of weakly contractive mapping we have%
\begin{eqnarray*}
\rho \left( u_{n+1},u_{n+p+1}\right) &=&\rho \left( Su_{n},Su_{n+p}\right) \\
&\leq &\rho \left( u_{n},u_{n+p}\right) -\varphi \left( \rho \left(
u_{n},u_{n+p}\right) \right)
\end{eqnarray*}%
for all $n,p\in 
\mathbb{N}
$ where $%
\mathbb{N}
$ is the set of pozitive integer. Set $\alpha _{n}=\rho \left(
u_{n},u_{n+p}\right) $. Since $\alpha _{n}$ is nonnegative and $\varphi $ is
nondecreasing function, we can write%
\begin{equation}
\alpha _{n+1}\leq \alpha _{n}-\varphi \left( \alpha _{n}\right) \leq \alpha
_{n}.  \label{m1}
\end{equation}%
This means that $\left\{ \alpha _{n}\right\} $ is a nonincreasing sequence.
Also, we know that $\left\{ \alpha _{n}\right\} $ is bounded below. Hence it
has a limit $\alpha \geq 0$.\ Assume that $\alpha >0$. Since $\varphi $ is
nondecreasing function we have%
\begin{equation*}
\varphi \left( \alpha _{n}\right) \geq \varphi \left( \alpha \right) >0.
\end{equation*}%
Therefore, from (\ref{m1}) we obtain%
\begin{equation*}
\alpha _{n+1}\leq \alpha _{n}-\varphi \left( \alpha \right) .
\end{equation*}%
Thus, we get $\alpha _{N+m}\leq \alpha _{m}-N\varphi \left( \alpha \right) $
which is a contradiction for large enough $N$. So, our assumption is wrong,
that is $\alpha =0$. Since $\alpha _{n}=\rho \left( u_{n},u_{n+p}\right)
\rightarrow 0$ as $n\rightarrow \infty $, $\left\{ u_{n}\right\} $ is a
Cauchy sequence in $E$. Hence, $\left\{ u_{n}\right\} $ converges to a point 
$u$ in $E$ because of the completeness of $E$. Now, we show that $\rho
\left( u,Su\right) =0$, i.e., $u$ is a fixed point of $S$. Using weak
contractivity of $S$ and definition of $b_{v}\left( s\right) $ metric $\rho $%
, we have%
\begin{eqnarray*}
\rho \left( u,Su\right) &\leq &s\left[ \rho \left( u,u_{n+1}\right) +\rho
\left( u_{n+1},u_{n+2}\right) +\ldots +\rho \left( u_{n+v-1},u_{n+v}\right)
+\rho \left( u_{n+v},Su\right) \right] \\
&=&s\left[ \rho \left( u,u_{n+1}\right) +\rho \left( u_{n+1},u_{n+2}\right)
+\ldots +\rho \left( u_{n+v-1},u_{n+v}\right) +\rho \left(
Su_{n+v-1},Su\right) \right] \\
&\leq &s\left[ \rho \left( u,u_{n+1}\right) +\rho \left(
u_{n+1},u_{n+2}\right) +\ldots +\rho \left( u_{n+v-1},u_{n+v}\right) \right.
\\
&&\left. +\rho \left( u_{n+v-1},u\right) -\varphi \left( \rho \left(
u_{n+v-1},u\right) \right) \right] .
\end{eqnarray*}%
Since $\rho \left( u_{n},u_{n+p}\right) \rightarrow 0$ and $u_{n}\rightarrow
u$ as $n\rightarrow \infty $ and $\varphi \left( 0\right) =0$, we obtain $%
\rho \left( u,Su\right) =0$, i.e., $u$ is a fixed point of $S$.

Now, we need to show that $u$ is a unique fixed point. Suppose to the
contrary that there exist a different fixed point $w$. Since $S$ is a weakly
contractive mapping, we have%
\begin{equation*}
\rho \left( u,w\right) =\rho \left( Su,Sw\right) \leq \rho \left( u,w\right)
-\varphi \left( \rho \left( u,w\right) \right) <\rho \left( u,w\right) .
\end{equation*}%
This is a contradiction. So $u=w$, that is $u$ is a unique fixed point of $S$%
. This completes the proof.
\end{proof}

\begin{remark}
In Theorem \ref{A},

\begin{enumerate}
\item if we take $v=s=1$ and $\varphi (t)=ct$, then we derive Banach fixed
point theorem.

\item if we take $\varphi (t)=ct$ then we derive Theorem 2.1 of \cite{mr}

\item if $v=1$ and $\varphi (t)=ct$, then we derive Theorem 2.1 of \cite{dun}%
.

\item if $v=2$ and $\varphi (t)=ct$, then we derive Theorem 2.1 of \cite{mit}
and so main theorem of \cite{geo}.

\item \i f $v=s=1$, then we derive main theorem of \cite{rho}.
\end{enumerate}
\end{remark}

Now, we prove Kannan fixed point theorem in $b_{v}\left( s\right) $ metric
spaces.

\begin{theorem}
\label{B}Let $E$ be a complete $b_{v}\left( s\right) $ metric space and $S$
a Kannan type mapping on $E$ such that $s\gamma \leq 1$. Then $S$ has a
unique fixed point.
\end{theorem}

\begin{proof}
Let $u_{0}\in E$ be an arbitrary initial point. Define sequence $\left\{
u_{n}\right\} $ by $u_{n+1}=Su_{n}$. If $u_{n}=u_{n+1}$, then it is clear
that $u_{n}$ is a fixed point of $S$. So, let us assume that $u_{n}\neq
u_{n+1}$ for all $n$. Moreover, we can assume that $x_{n}\neq x_{m}$ for all
different $n$ and $m$ in a similar way to \cite{geo}. Since $S$ is a Kannan
type mapping we have%
\begin{eqnarray*}
\rho \left( u_{n},u_{n+1}\right) &=&\rho \left( Su_{n-1},Su_{n}\right) \\
&\leq &\gamma \left[ \rho \left( u_{n-1},Su_{n-1}\right) +\rho \left(
u_{n},Su_{n}\right) \right] \\
&=&\gamma \left[ \rho \left( u_{n-1},Su_{n-1}\right) +\rho \left(
u_{n},u_{n+1}\right) \right]
\end{eqnarray*}%
and so,%
\begin{equation}
\rho \left( u_{n},u_{n+1}\right) \leq \frac{\gamma }{1-\gamma }\rho \left(
u_{n-1},Su_{n-1}\right) \leq \left( \frac{\gamma }{1-\gamma }\right)
^{n}\rho \left( u_{0},Su_{0}\right) .  \label{m2}
\end{equation}%
This implies that%
\begin{equation}
\lim_{n\rightarrow \infty }\rho \left( u_{n},u_{n+1}\right) =0.  \label{m3}
\end{equation}%
Using inequality (\ref{m2}), we get%
\begin{eqnarray*}
\rho \left( u_{n},u_{n+p}\right) &=&\rho \left(
S^{n}u_{0},S^{n+p}u_{0}\right) \\
&\leq &\gamma \left[ \rho \left( S^{n-1}u_{0},S^{n}u_{0}\right) +\rho \left(
S^{n+p-1}u_{0},S^{n+p}u_{0}\right) \right] \\
&\leq &\gamma \left[ \left( \frac{\gamma }{1-\gamma }\right) ^{n-1}\rho
\left( u_{0},Su_{0}\right) +\left( \frac{\gamma }{1-\gamma }\right)
^{n+p-1}\rho \left( u_{0},Su_{0}\right) \right] .
\end{eqnarray*}%
Since $\gamma \in \left[ 0,\frac{1}{2}\right) $, $\rho \left(
u_{n},u_{n+p}\right) \rightarrow 0$ as $n\rightarrow \infty $. It means that 
$\left\{ u_{n}\right\} $ is a Cauchy sequence in $E$. By completeness of $E$%
, there exists a point $u$ in $E$ such that $u_{n}\rightarrow u$. Now we
show that $u$ is a fixed point of $S$. For any $n\in 
\mathbb{N}
$, we have%
\begin{eqnarray*}
\rho \left( u,Su\right) &\leq &s\left[ \rho \left( u,u_{n+1}\right) +\rho
\left( u_{n+1},u_{n+2}\right) +\ldots +\rho \left( u_{n+v-1},u_{n+v}\right)
+\rho \left( u_{n+v},Su\right) \right] \\
&=&s\left[ \rho \left( u,u_{n+1}\right) +\rho \left( u_{n+1},u_{n+2}\right)
+\ldots +\rho \left( u_{n+v-1},u_{n+v}\right) +\rho \left(
Su_{n+v-1},Su\right) \right] \\
&\leq &s\left[ \rho \left( u,u_{n+1}\right) +\rho \left(
u_{n+1},u_{n+2}\right) +\ldots +\rho \left( u_{n+v-1},u_{n+v}\right) \right.
\\
&&\left. +\gamma \left[ \rho \left( u_{n+v-1},Su_{n+v-1}\right) +\rho \left(
u,Su\right) \right] \right] ,
\end{eqnarray*}%
and so%
\begin{equation*}
\left( 1-s\gamma \right) \rho \left( u,Su\right) \leq s\left[ \rho \left(
u,u_{n+1}\right) +\rho \left( u_{n+1},u_{n+2}\right) +\ldots +\left(
1+\gamma \right) \rho \left( u_{n+v-1},u_{n+v}\right) \right] .
\end{equation*}

Using inequality (\ref{m3}), we have $\rho \left( u,Su\right) =0$, i.e., $%
u\in F\left( S\right) $. Now we need to show that $u$ is a unique fixed
point. Suppose to the contrary that there exist a different fixed point $w$.
Since $S$ is a Kannan type mapping, we have%
\begin{equation*}
\rho \left( u,w\right) =\rho \left( Su,Sw\right) \leq \gamma \left[ \rho
\left( u,Su\right) +\rho \left( w,Sw\right) \right] =0.
\end{equation*}%
Therefore $u=w$ which is a contradiction. So, $u$ is a unique fixed point of 
$S$.
\end{proof}

\begin{remark}
In Theorem \ref{B},

\begin{enumerate}
\item if $v=s=1$, then we obtain Kannan fixed point theorem \cite{kan} in
complete metric spaces.

\item if $v=2$, then we derive Theorem 2.4 of \ \cite{geo}.

\item if $v=2$ and $s=1$, then we obtain main theorem of \cite{das} without
the assumption of orbitally completeness of the space and the main theorem
of \cite{ak}.
\end{enumerate}
\end{remark}

\end{document}